\documentclass[12pt]{amsart} 
\usepackage{amsmath, amssymb, amsfonts, amstext, amsthm, amscd}
\usepackage{mathabx}
\usepackage{tikz-cd}
\usetikzlibrary{decorations.pathmorphing}
\tikzset{squiggly/.style={decorate, decoration={zigzag,amplitude=0.5mm,segment length=2mm}}}
\usepackage[mathscr]{euscript}
\usepackage{enumerate}
\usepackage{tikz,color,soul}
\usepackage{hyperref}
\usepackage{a4wide}

\usepackage[left=2.5cm,top=2.5cm,bottom=3cm,right=2.5cm]{geometry}

\newtheorem{theorem}{Theorem}[section]

\theoremstyle{definition}

\newtheorem{remark}[theorem]{Remark}
\newtheorem{example}[theorem]{Example}

\newcommand{\Pe}{\mathcal{P}}

\theoremstyle{plain}
\numberwithin{equation}{section}
\newtheorem*{theorem*}{Theorem}

\newtheorem{lemma}[theorem]{Lemma}
\newtheorem{corollary}[theorem]{Corollary}

\theoremstyle{definition}

\numberwithin{equation}{section}

\title[Lefschetz Pencils on  surfaces and an application to Seshadri constants]{Lefschetz Pencils on  surfaces and an application to Seshadri constants}
 \author[K. Hanumanthu]{Krishna Hanumanthu}
\address{Chennai Mathematical Institute, H1 SIPCOT IT Park, Siruseri, Kelambakkam 603103, India}
 \email{krishna@cmi.ac.in}

 \author[A. J. Parameswaran]{A. J. Parameswaran}
 
 \address{Kerala School of Mathematics, Kunnamangalam PO, Kozhikode, Kerala, 673571, India}
 \email{param@ksom.res.in}



\date{\today}
\begin{document}
\begin{abstract}
Let $X$ be a nonsingular complex projective surface and let $L$ be a very ample line
bundle on $X$. We study the question: when does a general pencil in the linear system 
$|L|$ consist of only irreducible and reduced curves? We give a precise condition which guarantees an affirmative answer: in fact, this holds if and only if the genus of a general
member of $|L|$ exceeds the irregularity of $X$. The proof is topological, using
vanishing cycles and the theory of Lefschetz pencils. As an application, combined with a construction due to Miranda, we prove that for
any positive rational number $\delta$ there exist a blow up $Y \to X$ at finitely many
points, an ample line bundle $M$ on $Y$ and a point $x \in Y$ such that
$\varepsilon(Y,M,x) = \delta$; Miranda proved this when $X = \mathbb{P}^2$. 
\end{abstract}
\maketitle

\section{Introduction}\label{Intro}

Let $X$ be a smooth complex projective variety and let $L$ be a nef line bundle on
$X$.
The {\it Seshadri constant} of $L$ at $x \in X$ is defined as the real number
$$\varepsilon(X,L,x):=  \inf\limits_{\substack{x \in C}} \frac{L\cdot
  C}{{\rm mult}_{x}C},$$ where the infimum is taken over all
irreducible and reduced curves passing through $x$.
Seshadri constants were defined by Demailly in \cite{D},
motivated by  the Seshadri
criterion for ampleness (\cite[Theorem 7.1]{H})
which says that $L$ is ample if and only if $\varepsilon(X,L,x) > 0$
for all $x \in X$.

It is easy to see that Seshadri constants satisfy the upper bound $\varepsilon(X,L,x) \le \sqrt[n]{L^n}$, where $n$ is the dimension of $X$. 
As mentioned above, by the Seshadri criterion for ampleness, Seshadri constant of an ample line bundle at any point is positive. 
A lot of research is focused on finding better upper and lower bounds for Seshadri constants.

Ein-Lazarsfeld \cite{EL} gave non-trivial lower bounds at general points on surfaces. More precisely, they showed that if $L$ is an ample line bundle on a surface $X$, $\varepsilon(X,L,x) \ge 1$ at a \textit{very general} point $x$. On the other hand, Seshadri constants at arbitrary points can be arbitrarily small as shown by a construction of Miranda; see \cite[Example 3.1]{EL}. Miranda produces such examples on suitable blow ups of the projective plane $\mathbb{P}^2$. Bauer \cite[Proposition 3.3]{B} extended this construction to surfaces of Picard rank 1. 

Both these results use the same idea: start with a curve $C$ which has a point of sufficiently high multiplicity, find a pencil of irreducible and reduced curves containing $C$ and then blow up the surface at the base locus of the pencil. Then a suitably chosen line bundle will have the required small Seshadri constant. The key step is the construction of a pencil of irreducible and reduced curves. On $\mathbb{P}^2$ and more generally on surfaces of Picard number 1, this follows from a dimension count argument.  For arbitrary surfaces, such dimension count arguments do not easily generalize. 

In this note, we generalize the result to arbitrary surfaces. Since the dimension-count arguments used by Miranda and Bauer
no longer apply, we instead study the topology of hyperplane sections directly, using vanishing cycles and the theory of Lefschetz pencils to construct the necessary pencil. 
We answer the following general question: if $X$ is a smooth complex projective surface and $L$ is a very ample line bundle on $X$, when does the linear system $|L|$ contain a pencil of irreducible and reduced curves. 
The answer is given in Theorem \ref{main-top2}, which gives a topological criterion for irreducibility of singular members of Lefschetz pencils. More precisely, we show that the genus of a general member of $|L|$ is greater than the irregularity of $X$ if and only if a general pencil in $|L|$ contains only irreducible and reduced curves. Using this result and Miranda's construction we prove our main result about Seshadri constants in Theorem \ref{main-sc}.

Convex Fujita numbers of surfaces arise in the study of effective global generation of adjoint line bundles on surfaces. In \cite{CKMS}, several results on these numbers are proved. One such result uses pencils of irreducible and reduced curves. We briefly discuss the relevant result of \cite{CKMS} and  mention how we can establish further properties satisfied by the ample line bundles arising in this construction, in connection with convex Fujita numbers. 

Section \ref{bertini} proves a Bertini-type result which shows that 
any two hyperplane sections whose only singularity is a node are homeomorphic.
Section \ref{irreducible} uses vanishing cycles to give a topological criterion for a general pencil to contain only irreducible and reduced curves. We also discuss 
when this topological condition fails.
Section \ref{pencil} constructs pencils with desired properties containing a given curve. Section \ref{high-mult} constructs curves with given multiplicity at a point, and Section \ref{sc} applies the topological constructions of the previous sections and Miranda’s idea to obtain arbitrarily small Seshadri constants on blow ups of any surface.

\section{Bertini-type results for singular hyperplane sections} \label{bertini}

Let $X \subset \mathbb P^n$ be a nondegenerate, smooth complex projective surface. Assume that $n\ge 3$. We recall the definition of the dual variety and some of its properties below.

Let $\widecheck{\mathbb P^n}$ be the projective space of hyperplanes in $\mathbb P^n$. Consider the incidence variety
given by 
\begin{equation}\label{y}
Y := \{(x,H) \in \mathbb P^n \times \widecheck{\mathbb P^n}  \mid x \in X\cap H \}. 
\end{equation}

Then $Y$ is nonsingular since it is a projective bundle over $X$. The fibres of the projection $p_2\big|_Y: Y \to \widecheck{\mathbb P^n}$ are the hyperplane sections of $X$ in $\mathbb P^n$. So $Y$ is also referred to as the universal hyperplane section over $X$.

We also define  
$Z \subset Y$ by
\[ Z = \{(x,H) \in \mathbb P^n \times \widecheck{\mathbb P^n} \mid x \in X\cap H \text{ and } T_xX \subset T_x H \cong H\}.\] 
Here $T_xX$ denotes the tangent space of $X$ at $x$. Note that $Z$ is isomorphic to 
$\mathbb P(T_{\mathbb P^n}|_X/T_X)$, the projective normal bundle. Hence it is a nonsingular variety of dimension $n-1$. 

One can check that $Z$ is the critical locus of the map 
$p_2\big|_Y: Y \to \widecheck{\mathbb P^n}$; see \cite[\S 2]{L}.

The \textit{dual variety} $\widecheck{X}$ of $X$ is defined to be the image of $Z$ in  $\widecheck{\mathbb P^n}$ by the second projection. 
For $H \in \widecheck{\mathbb P^n}$, 
 $X \cap H$ is nonsingular if and only if  $H \notin \widecheck{X}$.

\begin{theorem}\label{dim} \cite[Corollary 7.4]{FL} or \cite[Proposition 3.1]{E}.
The dimension of $\widecheck{X}$ is $n-1$.
\end{theorem}

By Lefschetz double duality theorem \cite[2.2]{L}, the map $Z \to \widecheck{X}$ is birational. Let 





\begin{equation}\label{u,v}
\begin{aligned}
     V &:= \{(x,H) \in Z \mid \text{Sing}(X\cap H) = \{x\} \text{ and } x \text{ is a nodal singularity of } X\cap H \}, \text{ and} \\ U& :=p_2(V)
     \end{aligned}
 \end{equation}

Then $V$ is a non-empty open subset of $Z$ and $U$ is a non-empty open subset of $\widecheck{X}$. Further $p_2\big|_V: V \to U$ is an isomorphism and $V$ is the maximal open subset of $Z$ on which $p_2$ is an isomorphism. These statements follow from \cite[Lemma 4.9]{Lo}. 



Let $Y_U:=p_2^{-1}(U)$. 

The above set-up is summarized in the following diagram. 

\bigskip

\begin{tikzcd}[row sep=huge, column sep=huge]
  & & Y_U := p_2^{-1}(U) \arrow[d, hook] \arrow[ddl] & \\
  V \arrow[urr, hook] \arrow[r, hook] \arrow[rd, "\cong"', "p_2\big|_V"] & Z \arrow[r, hook, crossing over] \arrow[rd, "p_2\big|_Z", two heads, crossing over] & Y \arrow[r, hook] \arrow[rd, "p_2\big|_Y"] & \mathbb{P}^n \times \widecheck{\mathbb{P}}^n \arrow[d, "p_2"] \\
  & U \arrow[r, hook] & \widecheck{X} \arrow[r, hook] & \widecheck{\mathbb{P}}^n
\end{tikzcd}
\bigskip

Note that the fibres of the map $Y_U:=p_2^{-1}(U) \to U$ are singular along $V$ and in fact they have nodal singularities along $V$. 
In an appropriate coordinate system on $Y$ the singularity of a fibre of $Y_U\to U$ is locally given by $xy=0$ and $V$ is given by $x=y=0$. So the pair 
$(Y_U, V)$ is a Whitney stratification of the singular space $Y_U$. 

Note that the map $Y_U \to U$ is proper. Further the restriction 
$Y_U\setminus V \to U$ is submersive because every fibre is smooth and the map $V\to U$ is an isomorphism. So the map $Y_U \to U$ is a proper stratified submersion with respect to the Whitney stratification. 
This implies that $Y_U \to U$ is a locally trivial fibration by 
Thom's First Isotopy Lemma \cite[Proposition 11.1]{M}. 
Since $U$ is connected, all fibres 
of $Y_U \to U$ are homeomorphic.  

As a consequence of this discussion, we obtain the following. 

\begin{theorem}\label{main-top1}
Let $X$ be a nonsingular complex projective  surface embedded non-degenerately in $\mathbb P^n$. Let $X\cap H_1$ and $X \cap H_2$ be  hyperplane sections each of which has a unique singularity which is a node. Then $X\cap H_1$ and $X \cap H_2$ are homeomorphic. In particular, general singular hyperplane sections are homeomorphic. 

\end{theorem}

\begin{remark}
Theorem \ref{main-top1} holds in all dimensions as long as there is a hyperplane section with an isolated singularity, equivalently when $\widecheck{X}$ is a hypersurface in  $\widecheck{\mathbb{P}}^n$.   
\end{remark}

\section{Irreducibility of hyperplane sections in a general pencil}\label{irreducible}

In this section, we prove our main result. It gives 
a criterion for the existence of a pencil of reduced and irreducible curves in the linear system of a very ample divisor on a nonsingular surface. 

\begin{theorem}\label{main-top2}
Let $X$ be a nonsingular complex projective  surface embedded non-degenerately in $\mathbb P^n$, with $n \ge 3$.
Let $Y \subset \mathbb P^n \times \widecheck{\mathbb P^n}$ be as in \eqref{y}. 

Then there exists a subset $R$ of  $\widecheck{\mathbb P^n}$ of codimension at least 2 such that  the fibres of $p_2\big|_Y: Y \to \widecheck{\mathbb P^n}$ over points outside $R$ are reduced and irreducible if and only if the 
following inequality holds for a general hyperplane $H$ in $\mathbb P^n$:
\begin{equation}\label{ineq}
\text{dim } H_1(X\cap H, \mathbb{C}) > 
\text{dim } H_1(X, \mathbb{C}). 
\end{equation}

Moreover, when the condition \eqref{ineq} holds, $R$ can be chosen such that the fibres over points outside $R$ are either nonsingular or irreducible curves 
whose unique singularity is a node. 



\end{theorem}

\begin{proof}


By Theorem \ref{dim}, $\text{dim}(\widecheck{X})=n-1$.
Recall the non-empty open set $U \subset \widecheck{X}$ defined in \eqref{u,v}. 

Let $L$ be a general line in  $\widecheck{\mathbb P^n}$. Then 
$L$ intersects $\widecheck{X}$ transversally and 
$L \cap \widecheck{X} \subset U$. This can be seen by projecting from a general point outside $\widecheck{X}$. 

We claim that $R: = \widecheck{X}\setminus U$ satisfies the conclusions of the theorem.





Fix a point $H_0\in U$ and let $X_0 := X \cap H_0$ be the corresponding hyperplane
section so that $X_0$ has a unique singular point $y_0$, which is a node.

Since $R = \widecheck{X} \setminus U$ is closed in $\widecheck{\mathbb{P}}^n$ and
$H_0 \notin R$, we may choose a ball $B \subset \widecheck{\mathbb{P}}^n$ centered at
$H_0$, of radius small enough that $B \cap \widecheck{X} = B \cap U$ is given by $z_n=0$ where $z_1, z_2, \cdots z_n$ are the coordinates of the ball $B$. Let
$\mathcal{B} := p_2^{-1}(B) \subset Y$. Since $p_2\big|_Y$ is proper, the sets
$p_2^{-1}(B')$, for balls $B' \subseteq B$ centered at $H_0$, form a fundamental
system of neighbourhoods of $X_0$ in $Y$; shrinking $B$ if necessary, we may
therefore assume that $X_0$ is a strong deformation retract of $\mathcal{B}$; 
see \cite[5.3.1]{L}.

Choose $H_1 \in B \setminus \widecheck{X}$ and put $X_1 := X \cap H_1$; this is a
nonsingular curve, since $H_1$ does not lie on the dual variety. Following \cite{L}, we 
recall briefly how
the node of $X_0$ determines a distinguished  class of embedded circles
on $X_1$.

Since $H_0 \in U$ and $B \cap \widecheck{X} = B \cap U$, every fibre over
$B$ has at most one singular point, and that point is a node; moreover
$p_2\big|_V \colon V \to U$ is an isomorphism,  $B \cap \widecheck{X}$ is a ball 
of dimension $n-1$ and the family $\mathcal{B} \to B$ restricted to the curve $D:=\{z_1=z_2=\cdots z_{n-1}=0\}$ is a $1$-parameter smoothing of $X_0$ and the total space of this smoothing is a smooth surface.
Hence there are $r, \varepsilon > 0$ with $\varepsilon \ll r$, local analytic
coordinates $(x,y)$ on $p_2^{-1}(D)$ at $y_0\in X_0$ defined on the ball
$\{|x|^2+ |y|^2 < r^2\}$ centred at $y_0$,  such that near $y_0$ the family is given by
\[
  x^2+y^2 = t, \qquad |x|^2 + |y|^2 < r^2, \quad |t| < \varepsilon.
\]

We may assume that $H_1 \in D$. 
We may also assume that the coordinate $t$ on $D$ vanishes
at $H_0$, and the coordinate $t_1$ of $H_1$ is a positive real number. 
Thus $X_0$ is locally $\{x^2+y^2 = 0\}$, the union of the two branches $x = \pm i y$ where $i$ is a square root of $-1$.

The fibre $\{x^2+y^2 = t_1\}$ with $t_1$ positive real can be described as follows in terms of real coordinates. 
Let $x=u_1+iv_1$ and $y=u_2+iv_2$ be the real representatives of the complex variables. 
Then the fibre over a real number $t_1$ is expressed as
\[
\{ (u,v) \,\,|\,\,\,\, u_1^2-v_1^2+u_2^2-v_2^2 =t_1 \,\,\,\,\,\, u_1v_1+u_2v_2 =0 \,\,\,\,\,\,~u_1^2+v_1^2 + u_2^2 + v_2^2 < r^2 \}
 \]

When $\quad t_1 < \varepsilon \ll r$ it follows that this fibre is isomorphic to a disc bundle as follows: consider the coordinates $u=(u_1, u_2)$ and $v=(v_1,v_2)$. Then the fiber is 
identified as
\[
\{ (u,v)\,\,\, \mid \,\,\,|u|^2 - |v|^2 = t_1 \le |u|^2 + |v|^2 < r^2, \,\,\,and \,\,\, u\cdot v =0~~~\}
\]

This is an annulus, identified as the disc bundle of the tangent bundle of the circle in the coordinates $u$ and its tangent discs in the coordinates $v$. Hence it deformation retracts onto its core circle $|u|^2=t_1$ given by 

\[
  \delta_1 \;=\; \bigl\{ u_1^2+u_2^2=t_1;\,\,\, v_1=v_2=0 \bigr\}.
\]
This is an embedded circle in $X_1$ whose class generates the first homology of 
the annulus. We call $\delta_1$ the \emph{local vanishing cycle} at $y_0$: its radius
$\sqrt{t_1}$ tends to $0$ as $t_1 \to 0$, so $\delta_1$ contracts to the node.
In the total space, $\delta_1$ bounds the \emph{thimble}
\[
  \Delta_1 \;=\; \bigl\{ u_1^2+u_2^2=t;\,\,\, v_1=v_2=0; \,\,\, 0\leq t \leq t_1 \bigr\},
\]
where $t = 0$ gives the single point $y_0$.  $\Delta_1$ is the cone over
$\delta_1$ with apex $y_0$, swept out by the vanishing cycles as $t$ runs from
$t_1$ to $0$. Letting $t \to 0$ and using that $X_0$ is a deformation retract of
$\mathcal{B}$, we conclude that $X_0$ is homotopy equivalent to
$X_1 \cup_{\delta_1} \Delta_1$, i.e.,\ $X_0$ is obtained from $X_1$ by attaching the
$2$-cell $\Delta_1$ along the cycle $\delta_1$; see \cite[Figure~4]{L} for a pictorial representation of this phenomenon. 
More details about this discussion can be found in 
\cite[\S 5]{L}.







We now prove the following claim which relates the vanishing cycles to the irreducibility of $X\cap H_0$. 

{\bf Claim:} The class of the cycle  $\delta_1$ is nonzero in $H_1(X\cap H_1, \mathbb C)$ if and only if 
$X\cap H_0$ is irreducible. 

{\bf Proof of Claim:} 
For simplicity, let $X_0:=X\cap H_0$ and $X_1:=X\cap H_1$. By the above discussion, $X_0$ is obtained 
from $X_1$
by attaching the 2-cell $\Delta_1$  along the local vanishing cycle $\delta_1$ and $X_0$ is homotopic to $X_1 \cup_{\delta_1} \Delta_1$.

By excision the homologies of the pair  $(X_1\cup \Delta_1,X_1)$ are isomorphic to 
the homologies of the pair $(\Delta_1,\delta_1)$, which in turn are isomorphic to the homologies of the pair $(D^2,S^1)$.  

Hence $H_2(X_1\cup \Delta_1,X_1) \cong  H_2(\Delta_1,\delta_1)$ (by excision) is a $\mathbb{C}$-vector space of dimension one generated by $(\Delta_1, \delta_1)$. Moreover, for $i\ne 2$,
 $H_i(X_1\cup \Delta_1,X_1)=0$.

Considering  the long exact sequence associated to the pair $(X_1\cup \Delta_1,X_1)$, we obtain 
\[ 0 \to H_2(X_1, \mathbb C) \to H_2(X_0, \mathbb C) \to H_2(X_1\cup \Delta_1,X_1) \cong \mathbb C \to H_1(X_1, \mathbb C) \to H_1(X_0, \mathbb C) \to 0. \]

Note that 
the boundary map $ H_2(X_1\cup \Delta_1,X_1) \cong \mathbb C \to H_1(X_1, \mathbb C)$ sends the generator $(\Delta_1, \delta_1)$ to  $\delta_1$. 

Now if $X_0$ is reducible then  $H_2(X_0, \mathbb C)$ has dimension at least $2$. On the other hand  $H_2(X_0, \mathbb C)$  fits in an exact sequence with initial term  $H_2(X_1, \mathbb C) \cong \mathbb C$ and the next term $H_2(X_1\cup \Delta_1,X_1) \cong\mathbb C$, hence has dimension at most $2$. This implies  that $X_0$ is reducible if and only if   $H_2(X_0, \mathbb C)$ has dimension $2$, i.e., $X_0$ has exactly $2$ irreducible components. Now  $H_2(X_0, \mathbb C) \to \mathbb C$ is surjective as dimension of $H_2(X_0, \mathbb C)$ is $2$ and the kernel has dimension $1$.  Hence the class of $\delta_1$ must be zero in $H_1(X_1, \mathbb C)$ by the exactness of the sequence. 

Conversely, if the class of $\delta_1$ is zero in $H_1(X_1, \mathbb C)$, then the map $\mathbb C\to H_1(X_1, \mathbb C)$ is the zero map. Hence the dimension of $H_2(X_0, \mathbb C)$ is 2 and $X_0$ is reducible. 
So the claim is proved. 

Now recall that,
by Theorem \ref{main-top1}, the hyperplane sections corresponding to points in $U$ are either all irreducible or all reducible, as they are homeomorphic to each other. 
So by the above claim, the classes of vanishing cycles in $H_1(X\cap H_1, \mathbb C)$ 
are all zero or all non-zero. 

Let $C$ be a general fibre of $p_2$. We can transport the local vanishing cycles to $C$ and this induces isomorphisms on $H_1$, so whether a homology class is zero does not depend on the choices made.

By \cite[\S 7.2]{L}, the kernel of the map $H_1(C,\mathbb C) \to H_1(X, \mathbb C)$ is generated by local vanishing cycles. 
By the Lefschetz hyperplane theorem, the 
map $H_1(C,\mathbb C) \to H_1(X, \mathbb C)$ is surjective. 
Now we have the following:

\begin{align*}
  \text{The fibres over } U \text{ are all reducible}
    \;&\overset{\text{Claim}}{\Longleftrightarrow}\;
    \begin{aligned}[t]
      & \text{classes of vanishing cycles in } H_1(C, \mathbb C)\\
      &\text{are all zero}
    \end{aligned} \\[1ex]
    &\overset{\text{\cite[\S 7.2]{L}}}{\Longleftrightarrow}\;
    \text{kernel of } H_1(C,\mathbb C) \to H_1(X, \mathbb C)  \text{ is zero} \\
    &\Longleftrightarrow\;
    \dim H_1(C,\mathbb{C}) = \dim H_1(X,\mathbb{C}).
\end{align*}

Thus if $\dim H_1(C,\mathbb{C}) > \dim H_1(X,\mathbb{C})$, then all the fibres over $U$ are irreducible. By the construction in Section \ref{bertini}, fibres over $U$ have a unique singularity and it is a node. So 
$R = \widecheck{X}\setminus U$ has the required properties. 

Conversely, suppose that  $\dim H_1(C,\mathbb{C}) = \dim H_1(X,\mathbb{C})$. Then the fibres over $U$ are all reducible. 
Note that $U$ is a dense open subset of $\widecheck{X}$ and
$\widecheck{X}$ is a hypersurface in $\widecheck{\mathbb P^n}$ and 
if $R$ is any subset of $\widecheck{\mathbb P^n}$ of codimension at least 2, then $R$ cannot contain $U$. So the complement of $R$ contains a point of $U$ and the fibre over this point is reducible.  Consequently $R$ does not have the desired property. 

This completes the proof.
\end{proof}




\begin{remark}
A version of the statement of Theorem \ref{main-top2} has been proved in \cite[Proof of Theorem 7.5]{CKMS}.   
In that paper it is shown that for a given very ample line bundle $L$ on a smooth complex projective surface $X$, a suitable multiple of $L$ contains a pencil of 
irreducible and reduced curves
(for their argument, $3L$ works). In order to prove their result, they use the notion of 2-connected divisors. We note that our proof of Theorem \ref{main-top2} is very different as it uses topological methods. Moreover, we give an explicit criterion on $L$ that ensures the existence of such pencils. In remarks below, we look at pairs $(X,L)$ for which the required condition fails. 

For more details on how the pencils of irreducible curves are used in \cite{CKMS}, see Remark \ref{fujita}. 
\end{remark}

\begin{example}
Let $X=\mathbb{P}^2$ and $L = \mathcal{O}_X(2)$. Then a general pencil of conics in $\mathbb P^2$ contains a singular conic which must be union of two lines. Note that 
the inequality \eqref{ineq} fails for $(\mathbb{P}^2,\mathcal{O}_X(2))$. On the other hand, it is satisfied for $(\mathbb{P}^2,\mathcal{O}_X(d))$ when $d\ge 3$.
\end{example}

\begin{example}\label{3.3}
It is interesting to study pairs $(X,L)$ for which the 
inequality \eqref{ineq}
fails. In other words, the following natural map is an isomorphism for a general curve $C \in |L|$: 
\begin{equation*}
H_1(C,\mathbb C) \to H_1(X,\mathbb C).    
\end{equation*}

This is the case, for example, when $X= \mathbb P^2$ and $L = \mathcal{O}_X(1)$ or
 $L = \mathcal{O}_X(2)$. Another example is when $X$ is a smooth quadric in $\mathbb P^3$ and $L$ is the pullback of 
 $\mathcal{O}_{\mathbb P^3}(1)$.  
 
 Also if $B$ is a curve and
$V$ is a rank-2 vector bundle on $B$ then $X= {\mathbb P}(V)$ and 
$L = \mathcal{O}_X(1)\otimes A$, where $A$ is sufficiently  ample  line bundle on $B$ will provide another example.
\end{example}

\begin{remark} \label{3.4}
The inequality \eqref{ineq} holds for many pairs $(X,L)$ of surfaces of $X$ and very ample line bundles $L$ on $X$. For instance, consider surfaces $X$ of non-negative Kodaira dimension $(\kappa(X) \ge 0)$ such that $q:=h^1(X,\mathcal{O}_X) \le 1$. 
Note that $\text{dim }H_1(X,\mathbb C)=2q$.
Let $L$ be a very ample line bundle on $X$ and let $C \in |L|$ be a general curve. Let $g$ be the genus of $C$.  Then $\text{dim }H_1(C,\mathbb C) = 2g$. 

By the adjunction formula, we obtain 
\[ C^2+K_X\cdot C = 2g-2\]
\end{remark}
Since $C$ is very ample and $mK_X$ is effective for $m\gg 0$, $C^2+K_X\cdot C > 0$. Hence $g\ge 2 > q$. 
So the inequality \eqref{ineq} holds for any surface with $q\le 1$ and 
$\kappa(X) \ge 0$.

Examples of surfaces covered by Remark \ref{3.4} 
include K3 surfaces, Enriques surfaces, and hypersurfaces in $\mathbb{P}^3$ of degree at least 4. 

\begin{remark}\label{3.5}
We can extend the argument in Remark \ref{3.4} to surfaces $X$ with $q=2$ and $\kappa(X) \ge 0$. Let $L$ be a very ample line bundle on $X$. As before, we have 
\[ L^2+K_X\cdot L = 2g-2.\]
We claim that $g > q=2$, as required. 

We first show that $L^2 \ge 3$. Indeed, $L^2$ is the degree of the embedding of $X$ in a projective space $\mathbb P^n$ determined by $L$. We have $L^2 \ge 2$ since this embedding is nondegenerate. In fact, it is well-known that the degree is at least $n-2+1=n-1$. So if $L^2=2$, then $n=3$. So $X$ is a quadric in $\mathbb P^3$, which is not possible since we assumed that $q=2$. Note that the case of a quadric in $\mathbb P^3$ has been covered by Example \ref{3.3}. Hence $L^2 \ge 3$, as claimed. 

Now since $mK_X$ is effective for $m \gg 0$, we have \[2g-2= L^2+K_X\cdot L \ge L^2 \ge 3.\] Hence $g\ge 3$, as required. 
\end{remark}
Examples of surfaces covered by Remark \ref{3.5} include abelian surfaces. 
\begin{remark}
If $L$ is ample, but not very ample it is possible that \eqref{ineq} fails outside the cases covered in Example \ref{3.3}. For example, let $X = J(C)$ be the Jacobian of a genus 2 curve $C$ and let $\Theta$ be the theta divisor on $X$. Then $\Theta$ is ample, but not very ample. Note that 
$\Theta^2=2$ and $h^0(X, \Theta)=1$. 
Moreover, $C$ is isomorphic to the unique curve in the linear system $|\Theta|$. So in this case, we get $g=q=2$. 

\end{remark} 

\begin{remark}
    In fact, all {\it quasi-polarized} surfaces $(X,L)$ (i.e., $L$ is a nef and big divisor on a smooth complex projective surface $X$) which satisfy the equality $g=q$ have been completely classified: for the case $\kappa(X) \le  1$ see \cite[Theorem 3.1]{F2} and for the case $\kappa(X)=2$ see \cite[Theorem 4.2]{F1}. 

In this classification the only cases where $L$ is very ample are the ones listed in Example \ref{3.3}.
 \end{remark}
 


\begin{remark}
Suppose that $X$ is a smooth complex projective variety of dimension $n \ge 3$ and let $L$ be a very ample line bundle on $X$. Let $Y \in |L|$ be a member of the linear system of $L$. Then a simple dimension count shows that if $Y$ is reducible then the singular locus of $Y$ has dimension at least 1: if $Y=Y_1\cup Y_2$ where $Y_1$ and $Y_2$ have same dimension as $Y$, then the dimension of $Y_1 \cap Y_2$ is at least $n-1+n-1-n = n-2 \ge 1$. Clearly $Y_1\cap Y_2$ is contained in the singular locus of $Y$. 

On the other hand, it is a standard fact that members of a general pencil of a very ample line bundle on a smooth variety have at worst isolated singularities. So on a smooth variety of dimension at least 3, a general pencil of hyperplane sections consists of irreducible and reduced members. 
\end{remark}

 \section{Construction of a general pencil containing a given curve}\label{pencil}
We now show that we can construct a pencil of reduced and irreducible curves containing a given curve. 

\begin{theorem}\label{main-top}
Let $X$ be a nonsingular complex projective surface embedded non-degenerately in $\mathbb P^n$.
Suppose that the following inequality holds for a general hyperplane $H$ in $\mathbb P^n$:
\[ h^1(X\cap H, \mathcal{O}_{X\cap H}) > 
h^1(X, \mathcal{O}_{X}). 
\]
Let $C \subset X$ be any hyperplane section. Then there exists a pencil $\mathcal{P}$ of curves on $X$ satisfying the following: 
\begin{enumerate}
\item $C \in \mathcal{P}$. 
\item Every curve in $\mathcal{P}$ different from $C$ is irreducible and reduced. 
\end{enumerate}
Moreover, if $C$ is reduced then the pencil $\mathcal P$ can be chosen to additionally satisfy the following condition: 
\begin{enumerate}
    \item[(3)]  Any two distinct curves in $\mathcal{P}$ intersect transversally. 
\end{enumerate}
 
 \end{theorem} 
\begin{proof}
Let $p \in \widecheck{\mathbb P^n}$ be the point corresponding to $C$. 
Consider the subset $R \subset \widecheck{\mathbb P^n}$ constructed in Theorem \ref{main-top2}.  Note that $R$ has codimension at least 2. We claim that a general line in $ \widecheck{\mathbb P^n}$
through $p$ meets $R$ at most at $p$. 

Consider the projection from $p$:  $$\pi_p: \widecheck{\mathbb P^n} \setminus \{p\} \to \mathbb P^{n-1},$$ 
where the target space $\mathbb P^{n-1}$ is the projective space parametrizing the lines in $\widecheck{\mathbb P^n}$ through $p$. Let $\overline{R} \subset \mathbb P^{n-1}$ be the closure of the image of $R\setminus \{p\}$. Then $\overline{R}$ has codimension at least 1. Hence the complement of $\overline{R}$ in $\mathbb P^{n-1}$ is a nonempty open set. Choose a point in the complement and the corresponding line $L$ 
in $\widecheck{\mathbb P^n}$. Then $L$ passes through $p$ and does not intersect $R\setminus \{p\}$. 

So this gives the required pencil satisfying the conditions {\it (1)} and 
{\it (2)}. 

Suppose further that $C$ is reduced. Then a general hyperplane section meets $C$ transversally. So intersecting $C$ transversally is an open condition. Further if $C$ and $C'$ meet transversally, then any two distinct curves in the pencil spanned by $C$ and $C'$ also meet transversally. Hence a general line through $p$ in  
$ \widecheck{\mathbb P^n}$ corresponds to a pencil in $X$
satisfying the condition {\it (3)}. 
\end{proof}

\section{An irreducible curve with prescribed multiplicity at a point}\label{high-mult}
The following is well-known, but we include a proof for completeness. 
\begin{lemma} \label{existence}
Let $X$ be a nonsingular complex projective surface with a very ample line bundle $\mathcal{O}_X(1)$ and let $x \in X$ be a point. Let $m$ be a positive integer.  Then there exists an integer $d_0$ such that for all $d \ge d_0$, 
the linear system $|\mathcal{O}_X(d)|$ contains 
an irreducible and reduced curve $C$ with ${\rm mult}_{x}C= m$ and which is nonsingular on $X \setminus \{x\}$.
\end{lemma}
\begin{proof}
Let $ \mathcal{M}$ denote the ideal sheaf of $x$. Since $\mathcal{O}_X(1)$ is ample, there exists $d_0$ such that 
the sheaf $\mathcal{M}^m \otimes \mathcal{O}_X(d)$ is globally generated for $d \ge d_0$. 

Consider the blow up $\pi: \tilde{X} \to X$ of $X$ at $x$, with the exceptional divisor $E$. Consider the line bundle  $L_d: = \pi^{\star}\mathcal{O}_X(d)-mE$ on $\tilde{X}$.
Then there is a natural surjective map 
$$ \pi^{\star}\mathcal{M}^m \to \mathcal{O}_{\tilde{X}}(-mE).$$ Twisting by 
$\pi^{\star}\mathcal{O}_X(d)$, we obtain a surjective map 
$$ \pi^{\star}(\mathcal{M}^m \otimes\mathcal{O}_X(d)) \to L_d.$$
Since $\mathcal{M}^m \otimes\mathcal{O}_X(d)$ is globally generated, so is $\pi^{\star}(\mathcal{M}^m \otimes\mathcal{O}_X(d))$. Then $L_d$ is also globally generated for $d \ge d_0$ since it is a quotient of a globally generated sheaf.

We also note that $L_d$ is ample for $d > m$. This follows, for example, since 
the Seshadri constant $\varepsilon(X,\mathcal{O}_X(1),x) \ge 1$ as 
$\mathcal{O}_X(1)$ is very ample. 
By choosing $d_0 > m$ if necessary, we conclude that $L_d$ is ample and globally generated for $d \ge d_0$. 

Since $L_d$ is globally generated, by Bertini, a general curve $\tilde{C} \in |L_d|$ is smooth. Since $L_d$ is ample, every curve in the linear system $|L_d|$ is connected. So $\tilde{C}$ is smooth and connected and hence it is irreducible and reduced. 
Then $C = \pi(\tilde{C})$ has the required properties: $C$ is irreducible, reduced, and smooth outside $x$. Further, ${\rm mult}_{x}C= \tilde{C}\cdot E = m$. 
\end{proof}

 \section{Seshadri constants}\label{sc}

We now prove our main theorem about Seshadri constants. 

\begin{theorem}\label{main-sc}
Let $X$ be a nonsingular complex projective  surface and let $\delta$ be a positive real number. Then there exists a blow up $\pi: Y \to X$ of $X$ at finitely many points in $X$, an ample line bundle $L$ on $Y$ and 
 a point $x \in Y$ such that $\varepsilon(Y,L,x) < \delta$. 
\end{theorem} 
\begin{proof}
Choose an integer $m \ge 2$ such that $\frac{1}{m} < \delta$. Let $C \subset X$ be an irreducible and reduced curve containing a point $x$ of multiplicity $m$; such a curve exists by Lemma \ref{existence}.
More precisely, $C$ is a section of a very ample divisor
$\mathcal{O}_X(d)$ vanishing at $x$ with multiplicity $m$. 
By increasing $d$ if necessary, we may also assume that the inequality \eqref{ineq} holds. 
Now choose a pencil $\mathcal{P}$ which contains $C$ and satisfies the properties of  
Theorem \ref{main-top}. 

Let $C' \in \mathcal{P}$, $C' \ne C$ and let $C \cap C' = \{p_1,\ldots, p_r\}$. Note that $r=C\cdot C'$ and each $p_i$ is a simple point of $C$ and $C'$. 

Let $\pi: Y \to X$ be the blow up of $X$ at $\{p_1,\ldots, p_r\}$. Let $\widetilde{C}$ denote the strict transform of $C$ and let $E_1= \pi^{-1}(p_1)$.  Let $L = a\widetilde{C}+E_1$, $a\ge 2$. We claim that $L$ is ample and $\varepsilon(Y,L,x) \le \frac{1}{m} < \delta$. 

The rest of the proof  follows from Miranda's construction, but we include the details here for completeness. 

We use Nakai's criterion to show the ampleness of $L$. First note that $$L^2 = a^2\widetilde{C}^2+2a\widetilde{C}\cdot E_1+E_1^2 = 0+2a-1 > 0.$$
Now let $D\subset Y$ be any irreducible and reduced curve.  We show that $L \cdot D > 0$. 

Note that the strict transform $\widetilde{\Pe}$ of $\Pe$ is base point free and determines a morphism $\phi: Y \to \mathbb{P}^1$. 

If $D$ is a fibre of $\phi$ then $D$ is contained in a member of 
the pencil  $\widetilde{\Pe}$ and by hypothesis all curves in  $\widetilde{\Pe}$ are reduced and irreducible. So $D$ itself is inside  $\widetilde{\Pe}$. Then 
$L\cdot D = L \cdot  \widetilde{C} = 1.$

Suppose now that $D$ is not a fibre of $\phi$. Then $\phi(D) = \mathbb{P}^1$. 
Let $p \in \mathbb{P}^1$. Then $F=\phi^{-1}(p)$ is a member of the pencil  $\widetilde{\Pe}$. Since $D$ is not a fibre, $D$ and $F$ are distinct and have non-empty intersection. 
So $D \cdot F  = D \cdot  \widetilde{C} > 0$. We also claim that $D \cdot E_1 \ge -1$. Indeed, 
if $D \cdot E_1 < 0$, then $E_1$ is a component of $D$. Since $D$ is irreducible and reduced, $D=E_1$. Then $D \cdot E_1 = E_1^2 = -1$. 
This gives $L \cdot D = a D \cdot  \widetilde{C} + D \cdot E_1 > 0$. 

This concludes the proof that $L$ is ample.

Now note that $x \notin C \cap C' = \{p_1,\ldots,p_r\}$. By abuse of notation, we denote the point in $Y$ above $x$ also by $x$. 
Then 
$$\varepsilon(Y,L,x) \le \frac{L \cdot \widetilde{C}}{m} = \frac{1}{m}  < \delta,$$ as required. 
\end{proof} 

\begin{remark}\label{equality} In the situation of Theorem \ref{main-sc}, we note that one in fact has the equality  
$\varepsilon(Y,L,x) = \frac{1}{m}$, where $m$ is the multiplicity of $C$ at $x$.  
We prove this by showing that the Seshadri 
ratio computed by any curve passing through $x$, different from $\widetilde{C}$, is larger than $\frac{1}{m}$. 

Let $\widetilde{C} \ne D \subset Y$ be any irreducible and reduced curve containing $x$. Then we have $$\widetilde{C} \cdot D \ge m{\rm mult}_{x}D.$$ Further, since $x \in D$, $D$ is not equal to $E_1$. 
Then $$L \cdot D = a\widetilde{C} \cdot D + D\cdot E_1 \ge am({\rm mult}_{x}D).$$ Hence 
$$ \frac{L \cdot D}{{\rm mult}_{x}D} \ge am >  \frac{1}{m}.$$

\end{remark}
Since it is obvious that $\varepsilon(dL,x) = d\varepsilon(L,x)$ for 
any ample line bundle $L$ and any positive integer $d$, 
using Remark \ref{equality} we obtain:
\begin{corollary} 
Let $X$ be a nonsingular complex projective  surface and let $\delta$ be a positive \textit{rational} number. Then there exists a blow up $\pi: Y \to X$ of $X$ at finitely many points in $X$, an ample line bundle $L$ on $Y$ and 
 a point $x \in Y$ such that $\varepsilon(Y,L,x) =\delta$. 
\end{corollary}

\begin{remark} It is easy to see that if $L$ is ample and globally generated, then $\varepsilon(L,x) \ge 1$ for all points $x$. 
Let $m$ be any positive integer and let $X$ be any surface. The construction in the proof of Theorem \ref{main-sc} gives an ample line bundle $L$ on a suitable blow up of $X$ such that $kL$ is not globally generated for $k < m$. 
\end{remark}

\begin{remark}\label{fujita}
The notion of a \textit{convex Fujita number} is studied in \cite{CKMS}. Let $X$ be a smooth complex projective surface. The convex Fujita number $\text{Fu}(X)$ is defined to be the smallest non-negative integer $m$ such that for all $s\ge m$ and any ample line bundles
$L_1,\ldots, L_s$, the adjoint bundle $K_X+L_1+\ldots+L_s$ is globally generated. 
Using Reider's result \cite{R}, one can show that
$0 \le \text{Fu}(X) \le 3$ for any surface $X$. 

In \cite[Theorem 7.5]{CKMS}, the authors show that for any surface $X$, there is a suitable blow up $X'$ of $X$ at finitely many points such that $\text{Fu}(X')= 3$. Their proof critically uses the existence of a pencil of 
reduced and 
irreducible curves.

Let $C, C'$ be reduced and irreducible curves on $X$ such that the pencil spanned by them consists of reduced and irreducible curves meeting transversally. 
Let $C \cap C' =\{p_1,\ldots,p_r\}$. 
The authors consider the blow up $X' \to X$ at $\{p_2,\ldots,p_r\}$. Let $M$ be the strict transform of $C$. Then the authors prove that

\begin{enumerate}
\item $M$ is ample, 
\item $K_{X'}+2M$ is not globally generated. 
\end{enumerate}

As a consequence, $\text{Fu}(X')= 3$.

By using our result, we can obtain additional properties of $M$. As in the proof of Theorem \ref{main-sc}, given an integer $m\ge 2$, let $C$ be an irreducible and reduced curve which admits a point $x$ of multiplicity $m$. Then, again as before, choose a pencil of irreducible and reduced curves containing $C$. Now we blow up the base locus of the pencil after removing a point as in \cite{CKMS}. Then the line bundle $M$ (defined as the strict transform of $C$ as above) will have the following properties in addition to (1) and (2). 
\begin{enumerate}
\setcounter{enumi}{2}
\item $\varepsilon(X',M,x) = 1/m.$
\item $kM$ is not globally generated for $k< m$. 
\end{enumerate}

Note that $K_{X'}+3M$ is globally generated by \cite{R}.

\end{remark} 


\section*{Acknowledgements}
	
	The first author is supported by 
a grant from Infosys Foundation and 
ANRF MATRICS grant (ANRF/ARGM/2025/002374/MTR). 
The second author is supported by KIIFB funding. 
A part of this work was done when the second author visited Chennai Mathematical Institute as Infosys Chair Professor. 
We are grateful to Ankit Rai for pointing out the work of Y. Fukuma \cite{F1,F2}, 
and to Alex K\"{u}ronya for informing us about the paper \cite{CKMS}. We also thank Mike Roth for useful discussions.

\end{document}